\documentclass[12pt]{amsart}
\pdfoutput=1
\usepackage[margin=1in]{geometry}
\usepackage{amsthm,amssymb,amsmath,hyperref}
\usepackage{color, mdframed, tikz}
\usepackage{graphicx,enumerate}
\usepackage{cite}

\makeatletter
\def\l@section{\@tocline{1}{12pt plus2pt}{0pt}{}{\bfseries}}
\def\l@subsection{\@tocline{2}{0pt}{2pc}{2pc}{}}
\makeatother
\makeatletter
\def\subsection{\@startsection{subsection}{2}{\z@}%
	{-3.25ex\@plus -1ex \@minus -.2ex}%
	{1.5ex \@plus .2ex}%
	{\normalfont\bfseries\boldmath}}
\def\subsubsection{\@startsection{subsubsection}{3}%
	\z@{.5\linespacing\@plus.7\linespacing}{-.5em}%
	{\normalfont\bfseries\boldmath}}
\renewcommand\paragraph{\@startsection{paragraph}{4}{\z@}%
	{3.25ex \@plus1ex \@minus.2ex}%
	{-1em}%
	{\normalfont\normalsize\bfseries}}

\allowdisplaybreaks

\theoremstyle{plain}
\newtheorem{thm}{Theorem}[section]

\newtheorem{cor}[thm]{Corollary}
\newtheorem{lem}[thm]{Lemma}
\newtheorem{prop}[thm]{Proposition}

\theoremstyle{definition}
\newtheorem{defn}[thm]{Definition}

\theoremstyle{remark}
\newtheorem{rem}[thm]{Remark}

\theoremstyle{plain}

\numberwithin{equation}{section}

\newtheorem{thmA}{Theorem}

\theoremstyle{plain} 
\newcommand{\thistheoremname}{}
\newtheorem{genericthm}[thm]{\thistheoremname}

  \newtheorem*{genericthm*}{\thistheoremname}
\newenvironment{namedthm*}[1]
  {\renewcommand{\thistheoremname}{#1}%
   \begin{genericthm*}}
  {\end{genericthm*}}

\newcommand{\eps}{\varepsilon}

\newcommand{\D}{{\mathbb D}}

\newcommand{\R}{{\mathbb R}}

\newcommand{\N}{{\mathbb N}}

\newcommand{\Z}{{\mathbb Z}}

\newcommand{\dist}{\hbox{ \rm dist}}

\newcommand{\calA}{{\mathcal A}}
\newcommand{\calB}{{\mathcal B}}
\newcommand{\calC}{{\mathcal C}}
\newcommand{\calD}{{\mathcal D}}

\newcommand{\calF}{{\mathcal F}}
\newcommand{\calG}{{\mathcal G}}

\newcommand{\calQ}{{\mathcal Q}}
\newcommand{\calR}{{\mathcal R}}

\newcommand{\calT}{{\mathcal T}}

\newcommand{\supp}{{\textrm{supp}}}

\makeatletter
\newcommand{\vast}{\bBigg@{4}}
\newcommand{\Vast}{\bBigg@{5}}
\makeatother

\usepackage{scalerel}[2014/03/10]
\usepackage[usestackEOL]{stackengine}
\def\intavg{\,\ThisStyle{\ensurestackMath{%
    \stackinset{c}{0\LMpt}{c}{0\LMpt}{\SavedStyle-}{\SavedStyle\phantom{\int}}}%
    \setbox0=\hbox{$\SavedStyle\int\,$}\kern-\wd0}\int}

\def\udot#1{\ifmmode\oalign{$#1$\crcr\hidewidth.\hidewidth
    }\else\oalign{#1\crcr\hidewidth.\hidewidth}\fi}

\def\R{\mathbb{R}}
\def\Z{\mathbb{Z}}
\def\T{\mathbb{T}}

\def\beq{\begin{equation}}
\def\eeq{\end{equation}}
\makeatletter
\newcommand{\doublewidetilde}[1]{{%
  \mathpalette\double@widetilde{#1}%
}}
\newcommand{\double@widetilde}[2]{%
  \sbox\z@{$\m@th#1\widetilde{#2}$}%
  \ht\z@=.9\ht\z@
  \widetilde{\box\z@}%
}
\makeatother

\def\one{\mbox{1\hspace{-4.25pt}\fontsize{12}{14.4}\selectfont\textrm{1}}}

\usepackage{tikz}

\makeatletter
\def\@makefnmark{%
  \leavevmode
  \raise.9ex\hbox{\fontsize\sf@size\z@\normalfont\tiny\@thefnmark}}
\makeatother

\begin{document}
	
\title[BBM spaces over Carleson tents]{On the Bourgain--Brezis--Mironescu spaces over Carleson tents}

\author{\'Arp\'ad B\'enyi, Bingyang Hu and Xiaojing Zhou}

\address{\'Arp\'ad B\'enyi: Department of Mathematics, Western Washington University, 516 High Street, Bellingham, WA 98225,  U.S.A.}%
\email{benyia@wwu.edu}

\address{Bingyang Hu: Department of Mathematics, Auburn University, 221 Parker Hall, Auburn, AL 36849, U.S.A.}%
\email{bzh0108@auburn.edu}

\address{Xiaojing Zhou: Department of Mathematics, Auburn University, 221 Parker Hall, Auburn, AL 36849, U.S.A.}%
\email{xiz0003@auburn.edu}

\begin{abstract}
We introduce Carleson analogs of the Bourgain--Brezis--Mironescu spaces $B$ and $B_0$ by measuring mean oscillation over upper Carleson tents. For these spaces, denoted by $B_{\calC}^p$ and $B_{\calC,0}^p$, we prove two types of structural results. First, we show that they contain several natural classes of functions, including BMO/VMO--Carleson spaces, tent-space potential classes, and fractional Sobolev classes. Second, motivated by Zhu's structural theorem for BMO spaces induced by the Bergman metric, we establish decompositions of $B_{\calC}^p$ and $B_{\calC,0}^p$ into bounded-oscillation and bounded-average components. We then revisit the Bourgain--Brezis--Mironescu rigidity phenomenon in the Carleson setting. Although the direct rigidity statement fails for $B_{\calC,0}^p$, we introduce a natural $B_{\calC}^p$-trace and prove that the rigidity theorem survives at the level of traces.
\end{abstract}
\date{\today}


\maketitle

\tableofcontents

\section{Introduction}

This paper is motivated by the interaction between two fundamental themes in analysis.

\begin{enumerate}
\item The first is the real-variable theory of functions of bounded mean oscillation (BMO). Since its introduction by John and Nirenberg in their celebrated work \cite{JohnNirenberg1961}, BMO has become a central space in harmonic analysis, appearing in endpoint estimates, Carleson measure characterizations, the continuum decomposition of BMO relative to a Haar system, the Calder\'on reproducing formula, paraproducts, the $T1$ theory of Calder\'on--Zygmund operators, and much more; see, for instance, \cite[Chapters 3-4]{Grafakos2014} and \cite[Section 9.5]{MuscaluSchlag2013}.

\vspace{0.1cm}

\item The second is the theory of Carleson tents and tent spaces. Originating in the influential work of Coifman, Meyer and Stein \cite{CoifmanMeyerStein1985}, tent spaces provide a natural language for measuring boundary behavior through conical and Carleson-type regions. This viewpoint has become particularly useful in complex function theory and operator theory, especially in Carleson embedding problems and in BMO/VMO-type characterizations of Hankel-type operators; see, for example, \cite{BekolleBergerCoburnZhu1990,Zhu1992,CohnVerbitsky2000,PauZhaoZhu2016,HuZhou2025,Xia2002}.
\end{enumerate}

With these two themes in mind, we ask what kind of function spaces arise when mean oscillation is measured not over ordinary cubes, but over Carleson tents, and how the resulting Carleson formulation is related to the classical one. 

Our starting point is the Bourgain--Brezis--Mironescu (BBM) space $B$ and its little-oh counterpart $B_0$, introduced in \cite{BourgainBrezisMironescu2015} as a ``common roof'' for several regularity assumptions forcing integer-valued functions to be constant. The purpose of this paper is to develop their Carleson analogs, study their functional properties, and examine to what extent the BBM-rigidity phenomenon persists in the Carleson setting.

\medskip 

We begin by recalling several definitions. Throughout this work, $Q_0:=[0,1)^n$ denotes the unit cube in $\R^n$. For $f\in L^1(Q_0;\mathbb R)$, the space $\textrm{BMO}(Q_0)$ is equipped with the seminorm
$$
\|f\|_{\textrm{BMO}}
:=
\sup_{0<\varepsilon \le 1}\sup_{a\in Q_0}
\left\{M(f,Q_\varepsilon(a)): Q_\varepsilon(a)\subset Q_0\right\},
$$
where $Q_\varepsilon(a)$ denotes the $\varepsilon$-cube centered at $a$, and, for any cube $Q\subset Q_0$,
$$
M(f,Q):=\intavg_Q |f-f_Q|,
\qquad
f_Q:=\intavg_Q f.
$$
As usual, this seminorm becomes a norm after quotienting out the constant functions.

\vspace{0.1cm}

Having fixed the classical BMO seminorm, we now recall the BBM construction. In \cite{BourgainBrezisMironescu2015}, the authors introduced a new function space $B$, together with its little-oh counterpart $B_0$, as a unified framework for several apparently unrelated rigidity results for integer-valued functions. This construction is closely related to BMO: when $n=1$, the space $B$ coincides with $\textrm{BMO}(Q_0)$, while for $n\geq 2$ it strictly contains $\textrm{BMO}(Q_0)$. We now recall the definitions of $B$ and $B_0$.

Let $0<\varepsilon \le 1$, let $\calF(\varepsilon)$ be a collection of mutually disjoint cubes $Q\subseteq Q_0$ with length $\ell(Q)=\varepsilon$ such that $\# \calF(\varepsilon) \le \varepsilon^{1-n}$, and let
\[[f]_{\varepsilon}^B:=\sup_{\calF(\varepsilon)} \left\{ \varepsilon^{n-1} \sum_{Q \in \calF(\varepsilon)} M(f, Q) \right\}.\]
We define the space
\[B=\left\{f\in L^1(Q_0; \mathbb R): \sup_{0<\varepsilon \le 1} [f]_{\varepsilon}^B<+\infty \right\},\]
and its little-oh counterpart space
\[B_0=\left\{f\in B: \overline{\lim_{\varepsilon \to 0}}\,\,[f]_{\varepsilon}^B=0\right\}.\]

A central motivation for introducing $B$ and $B_0$ is the following deep rigidity theorem.

\begin{thmA}[{\cite[Theorem 4]{BourgainBrezisMironescu2015}}]\label{BBMMain}
Let $f$ be a $\Z$--valued function on $Q_0$ such that $f\in B_0$. Then $f$ is constant.
\end{thmA}

The \emph{two main goals} of the present paper are:
\begin{enumerate}
\item[$\bullet$] to introduce the Carleson analogs of the spaces $B$ and $B_0$, and to study their functional and structural properties, and

\item[$\bullet$] to investigate whether the rigidity phenomenon in Theorem~\ref{BBMMain}, namely the principle that integer-valued functions satisfying suitable regularity assumptions must be constant, continues to hold in the Carleson setting.
\end{enumerate}
The investigations in this work fit naturally into Feichtinger's ``ranking of function spaces'' framework; see \cite{Feichtinger2015}.

\vspace{0.1cm}

We begin by introducing the appropriate notions needed to define the versions of the BMO-Carleson spaces we alluded to above. In what follows, for any cube $Q \subseteq Q_0$, we let $T_Q^{up}$ denote the \emph{upper Carleson tent} associated to $Q$, defined by
$$
T_Q^{up}:=\left\{(x, t) \in \R^n \times (0, \infty): x \in Q, \;\frac{\ell(Q)}{2} \le t <\ell(Q) \right\}.
$$
Then, for any $1 \le p<\infty$, $Q \subseteq Q_0$, and $f \in L^p(Q_0 \times (0, 1); \R)$, we set
\begin{equation} \label{20260416eq02}
M_{\calC, p}(f, Q):=\left(\intavg_{T_Q^{up}} |f-f_{T_Q^{up}}|^p \right)^{1/p},
\end{equation} 
where 
$$f_{T_Q^{up}}:=\intavg_{T_Q^{up} } f=\frac{1}{|T_Q^{up}|}\int_{T_Q^{up}}f,$$
and
$$
M^*_{\calC, p}(f, Q):=\left(\intavg_{T_Q^{up}} \left( \intavg_{T_Q^{up}} |f(y)-f(z)| dV(z) \right)^p dV(y) \right)^{1/p};
$$
here, $dV$ denotes the standard volume measure on $\R^{n+1}$. Observe that 
\begin{equation} \label{20260502eq30}
M_{\calC, p}(f, Q) \le M^*_{\calC, p}(f, Q) \le 2^{1/p} M_{\calC, p}(f, Q).
\end{equation} 
Let $1 \le p<\infty$. We denote by $\textrm{BMO}^p_{\calC}(Q_0)$ the $\textrm{BMO}$-Carleson space which consists of all measurable functions $f$ on $Q_0 \times (0, 1)$ such that 
$$
\|f\|_{\textrm{BMO}^p_{\calC}}:=\sup_{Q \subseteq Q_0} M_{\calC, p}(f, Q)<+\infty. 
$$
Similarly, we let $\textrm{VMO}_{\calC}^p(Q_0)$ be the $\textrm{VMO}$-Carleson space of measurable functions on $Q_0 \times (0, 1)$ such that
$$
\lim_{\varepsilon \to 0} \sup_{Q \subset Q_0, \; \ell(Q)=\varepsilon} M_{\calC, p}(f, Q)=0. 
$$

Next, we define the Carleson analogs of the $B$-space and its little-oh version. We denote by $B_{\calC}^p$ the collection of all measurable functions $f \in L^p(Q_0 \times (0, 1); \R)$ such that 
$$
\left\|f \right\|_{B_{\calC}^p} :=\sup_{0<\varepsilon \le 1} [f]_{\varepsilon, p}<+\infty,
$$
where, for $0<\varepsilon \le 1$, we denoted
\begin{equation} \label{20260416eq01}
[f]_{\varepsilon, p}:=\sup_{\calF(\varepsilon)} \left\{ \varepsilon^{n-1} \sum_{Q \in \calF(\varepsilon)} M_{\calC, p}(f, Q) \right\}; 
\end{equation}
here, as in the BBM setting, $\calF(\varepsilon)$ denotes a collection of mutually disjoint cubes $Q$ with $\ell(Q)=\varepsilon$ such that $\# \calF(\varepsilon) \le \varepsilon^{1-n},$ and the supremum in \eqref{20260416eq01} is taken over all such collections. It is easy to check that $\| \cdot \|_{B_{\calC}^p}$ defines a semi-norm on $B_{\calC}^p$ and it becomes a norm on the equivalent classes $B_{\calC}^p / \R$. We also write $B_{\calC, 0}^p$ for the space of all measurable functions $f \in B_{\calC}^p$ with
$$
[f]_p:=\overline{\lim_{\varepsilon \to 0}}\,\, [f]_{\varepsilon, p}=0. 
$$
\begin{rem}
\begin{enumerate}
\item The oscillation in \eqref{20260416eq02} is inspired by the mean oscillation over Bergman balls of fixed radius. Such oscillations play a central role in the study of BMO-type spaces in the Bergman metric and in their operator-theoretic applications; see, for instance, \cite{BekolleBergerCoburnZhu1990, HuHuoLanzaniPalenciaWagner2024, Zhu1992, PauZhaoZhu2016}.

\item When $n=1$, the space $B_{\calC}^p$ is closely related to the corresponding Carleson-tent BMO space on the unit disc. More precisely, one may define $\textrm{BMO}_{\calC,p}(\D)$ by replacing cubes in \eqref{20260416eq01} with arcs $I\subset\T$ and using the associated upper Carleson tents. This type of space is naturally connected with the boundedness and compactness of Hankel-type operators, or equivalently with commutators involving the Bergman projection; see \cite{Li1992, Li1994, LiLuecking1994, BekolleBergerCoburnZhu1990,Zhu1992,PauZhaoZhu2016, Luecking1992}.

\item In the original BBM-construction over $Q_0$, the space $B$ is defined using the $L^1$ mean oscillation $M(f,Q)$. This choice is consistent with the classical definition of $\textrm{BMO}(Q_0)$ and the John--Nirenberg theory, which makes the corresponding $L^p$ oscillation norms equivalent. By contrast, in the Bergman/Carleson setting such a John--Nirenberg phenomenon is no longer available: the corresponding BMO spaces $\textrm{BMO}_{\calC,p}(\D)$ generally depend on the exponent $p$; see, for instance, \cite{Zhu1992}. This is one reason why, in the present Carleson setting, it is natural to consider the full family of spaces $B_{\calC}^p$.
\end{enumerate}
\end{rem}

The first part of the paper is devoted to two structural aspects of the spaces $B_{\calC}^p$ and $B_{\calC,0}^p$. The first one is real-variable in nature and identifies several natural classes of functions contained in $B_{\calC}^p$. More precisely, we show that $B_{\calC}^p$ contains:
\begin{enumerate}
    \item[$\bullet$] the BMO--Carleson space over $Q_0$, which may be viewed as an analog of $\textrm{BMO}_{\calC,p}(\D)$;
    \item[$\bullet$] functions whose gradients belong to a scale-invariant tent space over $Q_0$;
    \item[$\bullet$] a fractional Sobolev class on $Q_0\times(0,1)$.
\end{enumerate}
The corresponding little-oh inclusions also hold for $B_{\calC,0}^p$. These results may be regarded as Carleson counterparts of the structural results for the BBM-spaces $B$ and $B_0$; compare with \cite[Examples 1--3 and $1'$--$3'$]{BourgainBrezisMironescu2015}. We refer the reader to Section~\ref{Sec2} for the precise statements and proofs.

\vspace{0.1cm}

The second structural aspect comes from a different, more complex-analytic, point of view. Motivated by Zhu's structural theorem for Bergman-metric BMO spaces \cite{Zhu1992}, we prove a decomposition of the form
$$
B_{\calC}^p
=
\textnormal{a bounded-oscillation space}
+
\textnormal{a bounded-average space},
$$
with an analogous statement for $B_{\calC,0}^p$ after the appropriate vanishing modifications. Thus, the decomposition theorem extends Zhu's Bergman-metric picture to the BBM--Carleson setting; see Theorem~\ref{20260502thm41}.

\medskip 

In the second part of the paper, we return to the main motivation for introducing the Carleson variants of the $B$-spaces: whether the rigidity result in Theorem~\ref{BBMMain} remains valid in the Carleson setting. The answer is \emph{no}, as shown by the following simple observation. 

\begin{lem} \label{20260608lem01}
Let $1 \le p<+\infty$. There exists a non-constant $\Z$-valued function $f$ on $Q_0 \times (0, 1)$ such that $f \in B_{\calC, 0}^p$. 
\end{lem}

\begin{proof}
Define 
$$
f(x, t):=\one_{\{t>1/2\}}, \qquad (x, t) \in Q_0 \times (0, 1). 
$$
It is clear that $f$ is a non-constant $\Z$--valued function. To see that $f \in B_{\calC, 0}^p$, it suffices to notice that if we let $0<\varepsilon<1/2$ in the definition of $B_{\calC, 0}^p$, then $M_{\calC, p}(f, Q)=0$, as in this case $T_Q^{\textnormal{up}} \cap \supp f=\emptyset$. This clearly gives the desired claim.
\end{proof}

However, most interestingly, we hope to convince the reader that there is much more to this story: while the rigidity result fails at the level of $B_{\calC,0}^p$ itself, it \emph{does} hold in the $B_{\calC}^p$-trace sense; see Theorem~\ref{20260610BBMCarleson} and a careful development of the $B_{\calC}^p$-trace in Section~\ref{Sec4}.

\medskip 

The structure of the paper is as follows. In Section~\ref{Sec2}, we establish basic inclusion results for the Carleson spaces $B_{\calC}^p$ and $B_{\calC,0}^p$. In particular, we show that these spaces contain natural Carleson analogs of BMO and VMO, functions whose gradients belong to suitable scale-invariant tent spaces, and fractional Sobolev classes. In Section~\ref{Sec3}, we prove structural decomposition theorems for $B_{\calC}^p$ and $B_{\calC,0}^p$, showing that they split into bounded-oscillation and bounded-average components in the BBM--Carleson sense. Finally, in Section~\ref{Sec4}, we revisit the BBM-rigidity theorem in the Carleson setting. Although the direct analog fails for $B_{\calC,0}^p$, we introduce the $B_{\calC}^p$-trace and prove that the rigidity phenomenon survives at the level of traces.

\vspace{0.1cm}

\noindent {\bf Notations.} Throughout this paper, for $a ,b \in  \mathbb{R}$, $a\lesssim b$ means there exists a positive number $C$, which is independent of $a$ and $b$, such that $a\leq C\,b$. Moreover, if both $a \lesssim  b$ and $b\lesssim a$ hold, we write $a \simeq b$. 

\vspace{0.1cm}

\noindent{\bf Acknowledgments.} The first author gratefully acknowledges the Department of Mathematics and Statistics at Auburn University for hosting him during a stimulating research visit, where this project was initiated. The authors also thank the Department of Mathematics at the University of Alabama at Tuscaloosa for its hospitality during a subsequent visit, where further discussions contributed to the development of this work. The first author was supported by an AMS-Simons Research Enhancement Grant for PUI Faculty. The second author was supported by the Simons Travel Grant MPS-TSM-00007213.

\section{First structural results for the $B_{\calC}^p$--spaces}\label{Sec2}

The goal of this section is to collect some inclusions related to the $\textrm{BMO}$-Carleson spaces introduced above.

\begin{prop}\label{prop12} Let $1\le p, q<+\infty$. We have the following inclusions:

(a) If $p <q$, then $B_{\calC}^q \subsetneq B_{\calC}^p$.

(b) $\textrm{BMO}^p_{\calC}(Q_0) \subseteq B_{\calC}^p$, with continuous injection.
\end{prop}

\begin{proof}
(a) The inclusion $B_{\calC}^q \subseteq B_{\calC}^p$ is clear due to H\"older's inequality. To see that this inclusion is strict, let $f \in L^p(T_{Q_0}^{up}) \backslash L^q(T_{Q_0}^{up})$, where $T_{Q_0}^{up}:=Q_0 \times [1/2, 1)$ is the upper Carleson tent associated to $Q_0$. On the one hand, since
$$
\dist(\supp f, Q_0 \times \{0\}) \ge 1/2,
$$
we have that for any $Q \subseteq Q_0$ with $T_Q^{up} \cap \supp f \neq \emptyset$, $\ell(Q) \ge 1/2$. Therefore 
\begin{align*}
\|f\|_{B_{\calC}^p}&=\sup_{1/2 \le \varepsilon \le 1} [f]_{\varepsilon, p} =\sup_{1/2 \le \varepsilon \le 1} \left\{ \varepsilon^{n-1} \sum_{Q \in \calF(\varepsilon)} \left( \intavg_{T_{Q}^{up}} |f-f_{T_Q^{up}}|^p \right)^{1/p} \right\} \\
& \lesssim \sup_{1/2 \le \varepsilon \le 1} \left\{ \sum_{Q \in \calF(\varepsilon)} \left(\int_{T_Q^{up}} |f-f_{T_Q^{up}}|^p \right)^{1/p} \right\} \lesssim \left(\int_{T_{Q_0}^{up}}|f|^p \right)^{1/p}<+\infty,
\end{align*}
which implies $f \in B_{\calC}^p$; in the second to last estimate, we have used the disjointness of the cubes $Q \in \calF(\varepsilon)$, which implies the disjointness of their Carleson tents $\{T_{Q}^{up}\}_{Q \in \calF(\varepsilon)}$. On the other hand, by definition, 
\begin{equation} \label{20260417eq03}
\|f\|_{B_{\calC}^q} \ge 2^{1/q} \left(\int_{T_{Q_0}^{up}} |f-f_{T_{Q_0}^{up}}|^q \right)^{1/q}. 
\end{equation} 
Since $f \in L^p(T_{Q_0}^{up})$, $|f_{T_{Q_0}^{up}}|<+\infty$. This together with the assumption $f \notin L^q(T_{Q_0}^{up})$ yields that the integral on the righ hand-side of \eqref{20260417eq03} diverges to $+\infty$. Thus $f\not\in B_{\calC}^q$. 

(b) Follows immediately from the estimate $\|f\|_{B_{\calC}^p} \le \|f\|_{\textrm{BMO}^p_{\calC}}$.
\end{proof}

\begin{rem} 
Proposition~\ref{prop12} is consistent with the inclusion properties of $\textrm{BMO}_{\calC}^p(\D)$; see, for instance, \cite{Zhu1992}.
\end{rem} 

Before stating our next result, we recall the notion of \emph{tent space} over $Q_0$; see \cite{CoifmanMeyerStein1985}. Let $1\le p<+\infty$. We denote by $\calT_{1-\frac{p-1}{n}}^p(Q_0)$ the space of measurable functions $f$ on $Q_0 \times (0, 1)$ such that 
\begin{equation}\label{20260418eq01}
\left\|f\right\|_{\calT_{1-\frac{p-1}{n}}^p(Q_0)}:=\sup_{Q \subseteq Q_0} \left(\frac{1}{|Q|^{1-\frac{p-1}{n}}} \int_{T_Q^{up}} |f|^pdV \right)^{1/p}<+\infty. 
    \end{equation}
Similarly, we introduce the little-oh version of this tent space in a natural way. We let $\calT_{1-\frac{p-1}{n}, 0}^p(Q_0)$ be the space of measurable functions $f$ on $Q_0 \times (0, 1)$ such that 
\begin{equation}\label{20260418eq11}
\lim_{\varepsilon\to 0}\sup_{Q \subseteq Q_0, \ell(Q)=\varepsilon} \frac{1}{|Q|^{1-\frac{p-1}{n}}} \int_{T_Q^{up}} |f|^pdV=0. 
    \end{equation}
In particular, the analytic version of this tent space plays an important role in complex function theory;  for example, it features prominently in the $\calQ_p$ and $F(p, q, s)$ embedding problems into tent spaces; see \cite{HuZhou2025, Xiaojie2008}. 

\begin{rem}
When \(p=1\), conditions \eqref{20260418eq01} and \eqref{20260418eq11} can be reformulated as saying that the measure \(|f|\,dV\) is a Carleson measure and a vanishing Carleson measure, respectively.
\end{rem}

\begin{prop} \label{prop13}
Let $1 \le p<+\infty$. If $\nabla f \in \calT_{1-\frac{p-1}{n}}^p(Q_0)$, then $f \in B_{\calC}^p$.
\end{prop}

\begin{proof}
Let $Q \subseteq Q_0$ with $\ell(Q)=\varepsilon>0$. By Poincare's inequality, we get
\begin{align*}
M_{\calC, p}(f, Q) 
&\simeq \frac{1}{\varepsilon^{\frac{n+1}{p}}} \left\|f-f_{T_Q^{up}} \right\|_{L^p(T_Q^{up})}  \\
& \lesssim \frac{1}{\varepsilon^{\frac{n+1}{p}}} \cdot \varepsilon \cdot \left\|\nabla f \right\|_{L^p(T_Q^{up})}.
\end{align*}
Thus, for any collection of mutually disjoint $\varepsilon$-cubes $\calF(\varepsilon)$, one has
\begin{align} \label{20260419eq01}
\varepsilon^{n-1} \sum_{Q \in \calF(\varepsilon)} M_{\calC, p}(f, Q)
& \lesssim \varepsilon^{n-1} \sum_{Q \in \calF(\varepsilon)} \frac{1}{\varepsilon^{\frac{n+1}{p}}} \cdot \varepsilon \cdot \left\|\nabla f \right\|_{L^p(T_Q^{up})} \nonumber \\
&= \varepsilon^{n-\frac{n+1}{p}} \sum_{Q \in \calF(\varepsilon)} \left( \int_{T_Q^{up}} |\nabla f|^pdV \right)^{1/p} \nonumber  \\
& \le \varepsilon^{n-\frac{n+1}{p}}  \cdot \frac{1}{\varepsilon^{\frac{n-1}{p'}}} \left( \sum_{Q \in \calF(\varepsilon)}  \int_{T_Q^{up}} |\nabla f|^p dV \right)^{1/p} \\
&  \le \left\| |\nabla f|^p \right\|_{\calT_{1-\frac{p-1}{n}}^p(Q_0)} \cdot \varepsilon^{n-\frac{n+1}{p}} \cdot \frac{1}{\varepsilon^{\frac{n-1}{p'}}} \left( \sum_{Q \in \calF(\varepsilon)} |Q|^{1-\frac{p-1}{n}} \right)^{1/p} \nonumber  \\
& \lesssim \left\| |\nabla f|^p \right\|_{\calT_{1-\frac{p-1}{n}}^p(Q_0)}. \nonumber 
\end{align}
The desired result then follows by taking the supremum over $\calF(\varepsilon)$ on both sides of the above estimate. 
\end{proof}

Let $1 \le p, q<+\infty$ and let $f$ be a measurable function on $Q_0 \times (0, 1)$. We say that $f$ belongs to the fractional Sobolev space $W^{\frac{2}{pq}, pq}(Q_0 \times (0, 1))$ if 
$$
\int_{Q_0 \times (0, 1)} \int_{Q_0 \times (0, 1)} \frac{|f(y)-f(z)|^{pq}}{|y-z|^{n+3}}dV(y)dV(z)<+\infty.
$$

\begin{prop} \label{prop16}
Let $1 \le p, q<+\infty$. Then $W^{\frac{2}{pq}, pq}(Q_0 \times (0, 1))\subseteq B_{\calC}^p$.
\end{prop}

\begin{proof}
Without loss of generality, we may assume $q>1$; the case $q=1$ is easier. Let $Q \subseteq Q_0$ with $\ell(Q)=\varepsilon$. Using again H\"older's inequality, we have 
\begin{align*}
M_{\calC, p}^*(f, Q) 
& \le \left( \frac{1}{|T_Q^{up}|^2} \iint_{T_Q^{up} \times T_Q^{up}} |f(y)-f(z)|^p dV(y)dV(z) \right)^{1/p} \\
& \lesssim \frac{1}{\varepsilon^{\frac{2(n+1)}{p}}} \left( |\varepsilon|^{\frac{n+3}{q}} \iint_{T_Q^{up} \times T_Q^{up}} \frac{|f(y)-f(z)|^p}{|y-z|^{\frac{n+3}{q}}} dV(y)dV(z) \right)^{1/p} \\
& \lesssim  \frac{1}{\varepsilon^{\frac{2(n+1)}{p}}} \cdot |\varepsilon|^{\frac{n+3}{pq}}  \left(\iint_{T_Q^{up} \times T_Q^{up}} \frac{|f(y)-f(z)|^{pq}}{|y-z|^{n+3}} dV(y)dV(z) \right)^{\frac{1}{pq}} \cdot \left(|T_Q^{up}|^{\frac{2}{pq'}} \right) \\
& =\frac{\varepsilon^{\frac{n+3}{pq}} \cdot \varepsilon^{\frac{2(n+1)}{pq'}}}{\varepsilon^{\frac{2(n+1)}{p}}} \cdot \left(\iint_{T_Q^{up} \times T_Q^{up}} \frac{|f(y)-f(z)|^{pq}}{|y-z|^{n+3}} dV(y)dV(z) \right)^{\frac{1}{pq}}. 
\end{align*}
Therefore, for any collection $\calF(\varepsilon)$ of mutually disjoint $\varepsilon$--cubes, one has
\begin{align} \label{20260419eq02}
& \varepsilon^{n-1} \sum_{Q \in \calF(\varepsilon)} M_{\calC, p}^*(f, Q)  \nonumber \\
& \lesssim \varepsilon^{n-1} \cdot  \frac{\varepsilon^{\frac{n+3}{pq}} \cdot \varepsilon^{\frac{2(n+1)}{pq'}}}{\varepsilon^{\frac{2(n+1)}{p}}}  \cdot \sum_{Q \in \calF(\varepsilon)} \left(\iint_{T_Q^{up} \times T_Q^{up}} \frac{|f(y)-f(z)|^{pq}}{|y-z|^{n+3}} dV(y)dV(z) \right)^{\frac{1}{pq}}  \nonumber \\
& \le \varepsilon^{n-1} \cdot  \frac{\varepsilon^{\frac{n+3}{pq}} \cdot \varepsilon^{\frac{2(n+1)}{pq'}}}{\varepsilon^{\frac{2(n+1)}{p}}}  \cdot \frac{1}{\varepsilon^{\frac{n-1}{(pq)'}}} \cdot \left( \sum_{Q \in \calF(\varepsilon)} \iint_{T_Q^{up} \times T_Q^{up}} \frac{|f(y)-f(z)|^{pq}}{|y-z|^{n+3}} dV(y)dV(z)  \right)^{\frac{1}{pq}}  \\
& \le \left(\int_{Q_0 \times (0, 1)} \int_{Q_0 \times (0, 1)} \frac{|f(y)-f(z)|^{pq}}{|y-z|^{n+3}} dV(y)dV(z)  \right)^{\frac{1}{pq}}. \nonumber 
\end{align}
The desired results follows by taking the supremum over $\calF(\varepsilon)$ on both sides of the above estimate. 
\end{proof}

The next result states similar inclusions for the little-oh version of the Carleson $B$-space.

\begin{prop}\label{prop17} 
Let $1\le p, q<+\infty$.

(a) $\textrm{VMO}_{\calC}^p(Q_0) \subseteq B_{\calC, 0}^p$. Moreover, $\textrm{VMO}_{\calC}^p(Q_0)=B_{\calC, 0}^p$ if $n=1$.

(b) If $\nabla f \in \calT_{1-\frac{p-1}{n}, 0}^p(Q_0)$, then $f \in B_{\calC, 0}^p$. 

(c) $W^{\frac{2}{pq}, pq}(Q_0 \times (0, 1)) \subseteq B_{\calC, 0}^p$. 
\end{prop}

\begin{proof}
(a) follows trivially from the definitions. 

(b) follows immediately from the proof of Proposition \ref{prop13}. Indeed, in \eqref{20260419eq01}, the assumption implies that
\[
\frac{1}{|Q|^{1-\frac{p-1}{n}}}\int_{T_Q^{up}} |\nabla f|^p\, dV
\]
can be made arbitrarily small whenever $\varepsilon>0$ is chosen sufficiently small.

(c) is a consequence of the estimate \eqref{20260419eq02} and the fact that 
$$
\left| \bigcup_{Q \in \calF(\varepsilon)} T_Q^{up} \times T_{Q}^{up} \right| \le \varepsilon^{n+3}.
$$
\end{proof}

The following corollary follows from Propositions~\ref{prop12}-\ref{prop17}.

\begin{cor}
Let $1 \le p, q<+\infty$. Then, 
$$
\textrm{BMO}_{\calC}^p(Q_0)+\nabla^{-1} \calT_{1-\frac{p-1}{n}}^p(Q_0)+W^{\frac{2}{pq}, pq}(Q_0 \times (0, 1)) \subseteq B_{\calC}^p,$$
and
$$
\textrm{VMO}_{\calC}^p(Q_0)+\nabla^{-1} \calT_{1-\frac{p-1}{n}, 0}^p(Q_0)+W^{\frac{2}{pq}, pq}(Q_0 \times (0, 1)) \subseteq B_{\calC, 0}^p.
$$
\end{cor}

\bigskip 

\section{Structural decomposition of the $B_{\calC}^p$--spaces}\label{Sec3}

In \cite{BekolleBergerCoburnZhu1990}, B\'ekoll\'e-Berger-Coburn-Zhu introduced the BMO spaces on the open unit ball in $\mathbb C^n$ in terms of the Bergman metric and using the $L^2$-norm with respect to volume measure. Then Zhu \cite{Zhu1992} studied this BMO space defined via the volume $L^p$ integral, precisely describing its dependence in the Bergman metric on $p$ as well as providing a structural decomposition for it; see \cite[Theorem 5]{Zhu1992}. Motivated by these results, in this section we show that the space $B_{\calC}^p$ has a similar structural decomposition as $\textrm{BMO}_{\calC}^p(Q_0)$. 

We begin by defining the BBM counterparts of the spaces of bounded oscillation and bounded average over Carleson tents; see also \cite{Zhu1992} for the appropriate versions in Bergman setting.

\begin{defn}
Let $1 \le p<+\infty$ and $f \in L^p(Q_0 \times (0, 1); \R)$. We say that $f$ belongs to the \emph{BBM-bounded average space} $\textrm{BBA}_{\calC}^p(Q_0)$ if 
\begin{equation} \label{20260610eq33}
\|f\|_{\textrm{BBA}_{\calC}^p}:=\sup_{0<\varepsilon \le 1} \sup_{\calF(\varepsilon)} \left\{\varepsilon^{n-1}  \sum_{Q \in \calF(\varepsilon)} \left( \intavg_{T_Q^{up}} |f|^p \right)^{1/p} \right\}<+\infty. 
\end{equation} 
 We say that $f$ belongs to the \emph{BBM-bounded oscillation space} $\textrm{BBO}_{\calC}(Q_0)$ if
$$
\|f\|_{\textrm{BBO}_{\calC}}:=\sup_{0<\varepsilon \le 1} \sup_{\calF(\varepsilon)} \left\{\varepsilon^{n-1} \sum_{Q \in \calF(\varepsilon)} O_{\calC}(f, Q)   \right\}<+\infty; 
$$
here, for any $Q \subseteq Q_0$, $O_{\calC}(f, Q)$ denotes the local oscillation of $f$ on $T_Q^{up}$, defined by
\begin{equation} \label{20260502eq02}
O_{\calC}(f, Q):=\textnormal{ess sup}_{z, z' \in T_Q^{up}} |f(z)-f(z')|,
\end{equation}
while $\calF(\varepsilon)$ is the disjoint union of $\varepsilon$-cubes as in \eqref{20260416eq01}.
\end{defn}
Clearly, we can also define the little-oh versions of these spaces in a natural way. Namely, we say that a function $f$ belongs to $\textrm{BBA}_{\calC, 0}^p$  if $f\in \textrm{BBA}_{\calC}^p$ and 
\[\overline{\lim_{\varepsilon \to 0}}\,\,\sup_{\calF(\varepsilon)} \left\{\varepsilon^{n-1}  \sum_{Q \in \calF(\varepsilon)} \left( \intavg_{T_Q^{up}} |f|^p \right)^{1/p} \right\}=0,\]
while $f$ belongs to $\textrm{BBO}_{\calC, 0}$ if $f\in \textrm{BBO}_{\calC}$ and
\[\overline{\lim_{\varepsilon \to 0}}\,\,\sup_{\calF(\varepsilon)} \left\{\varepsilon^{n-1} \sum_{Q \in \calF(\varepsilon)} O_{\calC}(f, Q)   \right\}=0.\]
Our main result is the following.

\begin{thm} \label{20260502thm41}
Let $1 \le p <+\infty$. Then, $B_{\calC}^p(Q_0) = \textnormal{BBO}_{\calC}(Q_0)+\textnormal{BBA}_{\calC}^p(Q_0)$.
\end{thm}

The result follows from Proposition~\ref{mainprop1} and Proposition~\ref{mainprop2}, which we state and prove below. It is also worth mentioning that the structural decomposition stated above admits a natural little-oh version.

\begin{thm}\label{20260502thm410}
Let $1 \le p <+\infty$. Then, $B_{\calC, 0}^p(Q_0) = \textnormal{BBO}_{\calC, 0}(Q_0)+\textnormal{BBA}_{\calC, 0}^p(Q_0)$.
\end{thm}
The proof of Theorem~\ref{20260502thm410} follows from a straightforward modification of the proof of Theorem~\ref{20260502thm41} and is left to the interested reader.

\begin{prop}\label{mainprop1}
Let $1 \le p < +\infty$. Then,
\[
\textnormal{BBO}_{\calC}(Q_0)+\textnormal{BBA}_{\calC}^p(Q_0)
\subseteq B_{\calC}^p(Q_0).
\]
\end{prop}

\begin{proof}
We show each of the inclusions 
$$\textnormal{BBO}_{\calC}(Q_0) \subseteq B_{\calC}^p(Q_0)\,\,\,\, \text{and}\,\,\,\, \textnormal{BBA}_{\calC}^p(Q_0) \subseteq B_{\calC}^p(Q_0).$$

Observe first that, by \eqref{20260502eq30}, we have
$$
M_{\calC, p}(f, Q)  \simeq \left(\intavg_{T_Q^{up}} \left( \intavg_{T_Q^{up}} |f(y)-f(z)| dV(z) \right)^p dV(y) \right)^{1/p} \le O_{\calC}(f, Q),
$$
which gives the first inclusion.

Then, using H\"older's inequality, it is easy to see that 
$$
M_{\calC, p}(f, Q) \lesssim \left(\intavg_{T_Q^{up}} |f|^p \right)^{1/p}, 
$$
which clearly gives the second inclusion. 
\end{proof}

The remainder of this section is devoted to proving the reverse inclusion in Proposition~\ref{mainprop1}, stated below as Proposition~\ref{mainprop2}. An important ingredient in the second step of the proof of this proposition is contained in the following simple lemma; see also Xia's work \cite{Xia2002}.

\begin{lem}  \label{20260502lem01}
Let $1 \le p<+\infty$ and let $Q_1, Q_2 \subseteq Q_0$. If
$$
\ell(Q_1) \simeq \ell(Q_2) \qquad \textrm{and} \qquad |T_{Q_1}^{up} \cap T_{Q_2}^{up}| \gtrsim |T_{Q_1}^{up}|,
$$
then  
\begin{equation} \label{20260502eq03}
\left|f_{T_{Q_1}^{up}}-f_{T_{Q_2}^{up}} \right| \lesssim M_{C, p}(f, Q_1)+M_{C, p}(f, Q_2). 
\end{equation} 
\end{lem}

\begin{proof}
Let $E:=T_{Q_1}^{up} \cap T_{Q_2}^{up}$. Then
$$
\left|f_{T_{Q_1}^{up}}-f_{T_{Q_2}^{up}} \right| \le \left|f_{T_{Q_1}^{up}}-f_E \right|+\left|f_E-f_{T_{Q_2}^{up}} \right|.
$$
Note that 
\begin{align*}
\left|f_{T_{Q_1}^{up}}-f_E \right|
&=\frac{1}{|E|} \left| \int_E \left(f_{T_{Q_1}^{up}}-f(z)\right)dV(z) \right|  \\
& \le \frac{|T_{Q_1}^{up}|}{|E|} \cdot \frac{1}{|T_{Q_1}^{up}|} \int_{T_{Q_1}^{up}}\left|f_{T_{Q_1}^{up}}-f(z)\right|dV(z) \\
& \lesssim M_{\calC, p}(f, Q_1). 
\end{align*}
Similarly, one has
$$
\left|f_E-f_{T_{Q_2}^{up}} \right| \lesssim M_{\calC, p}(f, Q_2). 
$$
The desired claim \eqref{20260502eq03} follows by combining the two estimates above. 
\end{proof}

\begin{prop}\label{mainprop2}
Let $1 \le p < +\infty$. Then,
\[
B_{\calC}^p(Q_0)
\subseteq
\textnormal{BBO}_{\calC}(Q_0)+\textnormal{BBA}^p_{\calC}(Q_0).
\]
\end{prop}

\begin{proof}

 Let $f \in B_{\calC}^p(Q_0)$, and denote by $\calD(Q_0)$ the collection of all dyadic off-springs of $Q_0$. Consequently, for any $z \in Q_0 \times (0, 1)$, there exists a unique $R \in \calD(Q_0)$ such that $z \in T_R^{up}$, and we denote this dyadic off-spring by $R(z)$. Define now 
$$
f_1(z):=\intavg_{T_{R(z)}^{up}} f=f_{T_{R(z)}^{up}}, \quad \textrm{and} \quad f_2(z):=f(z)-f_1(z), \qquad z \in Q_0 \times (0, 1),
$$
so that $f=f_1+f_2$. We divide the rest of the proof into two steps.

\vspace{0.1in} 

\noindent \textit{Step I: \underline{We claim that $f_2 \in \textnormal{BBA}_{\calC}^p(Q_0)$.} }

Fix (arbitrary) $0<\varepsilon \le 1$ and disjoint union of $\varepsilon$-cubes $\calF(\varepsilon)$; we need to estimate the term 
$$
\calA_{\varepsilon, \calF(\varepsilon)}(f_2):=\varepsilon^{n-1} \sum_{Q \in \calF(\varepsilon)} \left( \intavg_{T_Q^{up}} |f_2|^p \right)^{1/p}. 
$$
Let $Q \in \calF(\varepsilon)$. Note first that if $Q \in \calD(Q_0)$, then 
$$
\left( \intavg_{T_Q^{up}} |f_2|^p \right)^{1/p}=M_{\calC, p}(f, Q). 
$$
If $Q \notin \calD(Q_0)$, then observe that there is a finite collection of dyadic cubes $R \in \calD(Q_0)$, denoted by $\calR(Q)$, such that
$$
T_Q^{up} \subseteq \bigcup_{R \in \calR(Q)} T_R^{up}, 
$$
with $\# \calR(Q) \lesssim_n 1$. Indeed, if we let $k \in \N$ be such that $2^{-k-1} \le \ell(Q)=\varepsilon<2^{-k}$, then it is easy to check that $\calR(Q)$ can be taken to be the collection of all dyadic off-springs in $\calD_k(Q_0) \cup \calD_{k+1}(Q_0)$ satisfying $R \cap Q \neq \emptyset$. Here, $\calD_k(Q_0)$ is the collection of all dyadic off-springs of $Q_0$ with sidelength $2^{-k}$. Therefore, for $Q \notin \calD(Q_0)$ one has 
\begin{align*}
\left( \intavg_{T_Q^{up}} |f_2|^p \right)^{1/p}
& \le \left( \frac{1}{|T_Q^{up}|} \sum_{R \in \calR(Q)} \int_{T_R^{up}}|f-f_{T_R^{up}}|^p  \right)^{1/p} \\
& \lesssim \left( \sum_{R \in \calR(Q)} \intavg_{T_R^{up}} |f-f_{T_R^{up}}|^p \right)^{1/p} \\
& \le \sum_{R \in \calR(Q)} M_{\calC, p}(f, R). 
\end{align*} 
Combining the two possible cases about an arbitrary $Q\in\calF(\varepsilon)$, we get
\begin{align} \label{20260502eq01}
\calA_{\varepsilon, \calF(\varepsilon)} & \lesssim \varepsilon^{n-1} \sum_{Q \in \calF(\varepsilon)} \sum_{R \in \calR(Q)} M_{\calC, p}(f, R) \nonumber \\
& = \varepsilon^{n-1} \sum_{Q \in \calF(\varepsilon)} \left[ \sum_{R \in \calR(Q) \cap \calD_k(Q_0)} M_{\calC, p}(f, R)+ \sum_{R \in \calR(Q) \cap \calD_{k+1}(Q_0)} M_{\calC, p}(f, R) \right] \nonumber \\
& \lesssim \underbrace{2^{-k(n-1)} \sum_{Q \in \calF(\varepsilon)} \sum_{R \in \calR(Q) \cap \calD_k(Q_0)} M_{\calC, p}(f, R)}_{S_1} \nonumber  \\
& \qquad +\underbrace{2^{-(k+1)(n-1)} \sum_{Q \in \calF(\varepsilon)} \sum_{R \in \calR(Q) \cap \calD_{k+1}(Q_0)} M_{\calC, p}(f, R)}_{S_2}. 
\end{align} 
Consider the first summation in \eqref{20260502eq01}. We observe that, for
distinct \(Q,Q'\in\calF(\varepsilon)\), it may happen that $\calR(Q)\cap \calR(Q')\cap \calD_k(Q_0)\neq \emptyset$. 
However, for each fixed \(R\in \calD_k(Q_0)\), since \(\ell(R)\simeq \ell(Q)\)
whenever \(R\in\calR(Q)\), and since the cubes in \(\calF(\varepsilon)\) are
pairwise disjoint, the cube \(R\) can belong to \(\calR(Q)\) for at most $\lesssim_n 1$ choices of \(Q\in\calF(\varepsilon)\). Therefore, one can divide the collection of dyadic cubes 
$$
\bigcup_{Q \in \calF(\varepsilon)} \{R\}_{R \in \calR(Q) \cap \calD(Q_0)} 
$$
into a union of a finite collection of dyadic cubes, denoted by $\calF_1(2^{-k}), \dots \calF_{C_n}(2^{-k})$, satisfying for any $j \in \{1, \dots C_n\}$,
\begin{enumerate}
    \item [$\bullet$] $\# \calF_j(2^{-k}) \lesssim \varepsilon^{1-n}$;
    \item [$\bullet$] $R \cap R'=\emptyset$ for any $R, R' \in \calF_j(2^{-k})$ (since both $R$ and $R'$ are dyadic). 
\end{enumerate}
Here, $C_n$ is an absolute dimensional constant. Therefore, we have 
\begin{equation*}
S_1\lesssim \sum_{j=1}^{C_n} \left(2^{-k(n-1)} \sum_{R \in \calF_j(2^{-k})} M_{\calC, p}(f, R) \right) \lesssim  \left\|f \right\|_{B_{\calC}^p}
\end{equation*}
Note that the same argument applies to the second summation $S_2$ in
\eqref{20260502eq01}. Hence,
\[
\calA_{\varepsilon,\calF(\varepsilon)}
\lesssim
\|f\|_{B_{\calC}^p},
\]
where we note that, crucially, the implicit constant depends only on \(n\). Therefore, 
$$
\left\|f_2 \right\|_{\textnormal{BBA}_{\calC}^p}=\sup_{0<\varepsilon \le 1} \sup_{\calF(\varepsilon)} \calA_{\varepsilon, \calF(\varepsilon)} \lesssim \|f\|_{B_{\calC}^p}. 
$$
Thus the claim in \emph{Step I} holds. 

\medskip 

\noindent\textit{Step II: \underline{We claim that $f_1 \in \textnormal{BBO}_{\calC}(Q_0)$.}} 

Again, take any $0<\varepsilon \le 1$ and any disjoint union of $\varepsilon$-cubes $\calF(\varepsilon)$. In this case, we will have to estimate the term 
$$
\calB_{\varepsilon, \calF(\varepsilon)}(f_1):=\varepsilon^{n-1} \sum_{Q \in \calF(\varepsilon)}O_{\calC}(f_1, Q), 
$$
where $O_{\calC}(f_1, Q)$ is defined as in \eqref{20260502eq02}. Recall that 
$$
\|f\|_{\textnormal{BBO}_{\calC}(Q_0)}=\sup_{0<\varepsilon \le 1} \sup_{\calF(\varepsilon)} \calB_{\varepsilon, \calF(\varepsilon)}(f_1).
$$
We follow a similar argument as in \textit{Step I}. Let $Q \in \calF(\varepsilon)$. If $Q \in \calD(Q_0)$, then 
\begin{align*}
O_{\calC}(f_1, Q)
&=\textnormal{ess sup}_{z, z' \in T_Q^{up}} \left| f_{T_{R(z)}^{up}}-f_{T_{R(z')}^{up}} \right|\\
&=\textnormal{ess sup}_{z, z' \in T_Q^{up}} \left| f_{T_Q^{up}}-f_{T_Q^{up}} \right|=0.
\end{align*}
Therefore, it suffices to consider the case when $Q \notin \calD(Q_0)$. The idea here is to explore the ``chain" structure between ``adjacent" upper Carleson tents, and this is where Lemma~\ref{20260502lem01} becomes useful.

In this case, take any \(z,z'\in T_Q^{up}\)
such that \(R(z)\neq R(z')\) (otherwise,
\(f_{T_{R(z)}^{up}}-f_{T_{R(z')}^{up}}=0\)). Again, let \(k\in\mathbb N\)
be such that \(2^{-k-1}\le \ell(Q)=\varepsilon<2^{-k}\). Then
\[
\ell(R(z)),\ell(R(z'))\in
\{2^{-k-2},\,2^{-k-1},\,2^{-k},\,2^{-k+1}\}.
\]
Therefore, one can cover
\[
T_{R(z)}^{up}\cup T_Q^{up}\cup T_{R(z')}^{up}
\]
by a finite collection of upper Carleson tents whose base cubes are dyadic
cubes belonging to \(\calD_{k+2}(Q_0)\), \(\calD_{k+1}(Q_0)\),
\(\calD_k(Q_0)\), or \(\calD_{k-1}(Q_0)\), with the cardinality of this
collection bounded above by an absolute dimensional constant. In particular, one can
find a chain of dyadic cubes
\[
R_1,\dots,R_N\subseteq Q_0,\qquad N\le C_n,
\]
such that
\begin{enumerate}
    \item \(R_1=R(z)\) and \(R_N=R(z')\);
    \item \((T_{R_i}^{up})^\circ\cap (T_{R_j}^{up})^\circ=\emptyset\), \; \(1\le i<j\le N\);
    \item \(\mathcal H^{n}(\overline{T_{R_j}^{up}}\cap \overline{T_{R_{j+1}}^{up}})>0\),
    \; \(1\le j\le N-1\).
\end{enumerate}
One possible way of choosing such a chain of dyadic cubes is as follows. Assume that
\(\ell(R(z))\ge \ell(R(z'))\), and denote by \(R(z,z')\) the unique dyadic
ancestor of \(R(z')\) such that \(\ell(R(z,z'))=\ell(R(z))\). Then
\begin{enumerate}
    \item[(a)] first, we choose a chain of dyadic cubes of side length \(\ell(R(z))\)
    connecting \(R(z)\) to \(R(z,z')\);
    \item[(b)] second, we connect \(R(z,z')\) to \(R(z')\) by taking all dyadic cubes
    lying between them, namely those contained in \(R(z,z')\) and containing \(R(z')\).
\end{enumerate}
Next, we refine the above collection as follows. Without loss of generality, we may assume $\ell(R_j) \ge \ell(R_{j+1}), \; j=0, \dots, N-1$. Then, when we pass from $R_j$ to $R_{j+1}$,  
\begin{enumerate}
    \item [(i)] if $\ell(R_j)=\ell(R_{j+1})$, since $\mathcal H^{n}(\overline{T_{R_j}^{up}}\cap \overline{T_{R_{j+1}}^{up}})>0$, then we can take a cube $G_j \notin \calD(Q_0)$ with 
    $$
    \ell(G_j)=\ell(R_j) \quad \textrm{and} \quad |T_{R_j}^{up} \cap T_{G_j}^{up}|=|T_{R_j}^{up} \cap T_{G_{j+1}}^{up}|=|T_{R_j}^{up}|/3;
    $$
    \item [(ii)] if $\ell(R_j)=2\ell(R_{j+1})$, using the assumption $\mathcal H^{n}(\overline{T_{R_j}^{up}}\cap \overline{T_{R_{j+1}}^{up}})>0$ again, we can take a cube $G_j \notin \calD(Q_0)$ with 
    $$
    \ell(G_j)=3\ell(R_j)/4, \quad |T_{G_j}^{up} \cap T_{R_j}^{up}|=3|T_{G_j}^{up}|/16,
    $$
    and
    $$
    |T_{G_j}^{up} \cap T_{R_{j+1}}^{up}|=|T_{R_{j+1}}^{up}|/2. 
    $$
\end{enumerate}
Consequently, the above construction gives us a chain of cubes 
\begin{equation} \label{20260502eq10}
R_1, G_1, R_2, G_2, \dots, R_{N-1}, G_{N-1}, R_N \subseteq Q_0
\end{equation} 
with 
\begin{enumerate}
    \item [(1)] $R_1=R(z)$ and $R_N=R(z')$;
    \item [(2)] $\ell(R_i) \simeq \ell(G_j)$ for any $1 \le i \le N$ and $1 \le j \le N-1$;
    \item [(3)] There exists some absolute constant $c>0$, such that $|T_{R_j}^{up} \cap T_{G_j}^{up}| \ge c|T_{R_j}^{up}|$ and $|T_{G_j}^{up} \cap T_{R_{j+1}}^{up}| \ge c|T_{R_{j+1}}^{up}|$. 
\end{enumerate}
We refer the reader to Figure~\ref{Fig2} for an example for such a chain of cubes. 

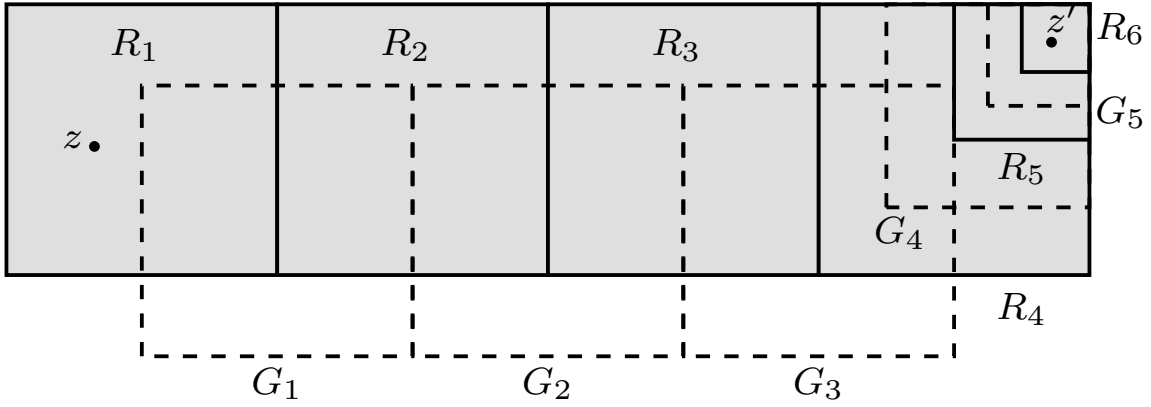
\begin{figure}[htbp]
\centering
\resizebox{0.96\textwidth}{!}{
\begin{tikzpicture}[every node/.style={font=\tiny}]

\tikzset{
  dyadic/.style={draw=black, thick, fill=gray!25},
  aux/.style={draw=black, thick, dashed},
  pt/.style={circle, fill=black, inner sep=.8pt}
}

\filldraw[dyadic] (0,0) rectangle (2,2);
\filldraw[dyadic] (2,0) rectangle (4,2);
\filldraw[dyadic] (4,0) rectangle (6,2);
\filldraw[dyadic] (6,0) rectangle (8,2);

\node at (0.95, 1.7) {$R_1$};
\node at (2.95, 1.7) {$R_2$};
\node at (4.95, 1.7) {$R_3$};
\node at (7.5, -0.25) {$R_4$};

\draw[aux] (1,-0.6) rectangle (3,1.4);
\draw[aux] (3,-0.6) rectangle (5,1.4);
\draw[aux] (5,-0.6) rectangle (7,1.4);

\node at (2,-.8) {$G_1$};
\node at (4,-.8) {$G_2$};
\node at (6,-.8) {$G_3$};

\filldraw[dyadic] (7,1) rectangle (8,2);

\filldraw[dyadic] (7.5,1.5) rectangle (8,2);

\draw[aux] (6.5,0.5) rectangle (8,2);
\draw[aux] (7.25,1.25) rectangle (8,2);


\node[below] at (7.5, 1.05) {$R_5$};

\node[right] at (7.9, 1.82) {$R_6$};

\node[below] at (6.6,0.57) {$G_4$};

\node[right] at (7.9, 1.2) {$G_5$};

\node[pt] at (0.65,0.95) {};
\node at (0.49,1) {$z$};

\node[pt] at (7.72,1.72) {};
\node at (7.8, 1.9) {$z'$};

\end{tikzpicture}
}
\caption{An example of a chain of cubes connecting $z$ and $z'$ with $N=6$: the shaded squares represent the cubes $R_i$, while the dashed squares represent the auxiliary cubes $G_i$.}
\label{Fig2}
\end{figure}

Thus, by Lemma \ref{20260502lem01}, we have 
\begin{align*}
|f_{T_{R(z)}^{up}}-f_{T_{R(z')}^{up}}|
&=|f_{T_{R_1}^{up}}-f_{T_{R_N}^{up}}| \\
& \le \sum_{j=1}^{N-1} \left|f_{T_{R_j}^{up}}-f_{T_{G_j}^{up}} \right|+\sum_{j=1}^{N-1} |f_{T_{G_j}^{up}}-f_{T_{R_{j+1}}^{up}}| \\
& \lesssim \sum_{j=1}^N M_{\calC, p}(f, R_j)+\sum_{j=1}^{N-1} M_{\calC, p}(f, G_j). 
\end{align*}
Observe that, in the above estimate, the choice of the chain
\eqref{20260502eq10} depends on the points $z$ and $z'$, or more precisely,
on the associated dyadic upper Carleson tents $R(z)$ and $R(z')$.
However, for each fixed $z \in T_Q^{up}$, the tent $R(z)$ has only finitely many choices; moreover, the number of such possibilities is
bounded by a constant depending only on the dimension $n$. This means that one can find a collection of dyadic cubes $\calR(Q)$ and a collection of auxiliary (non-dyadic) cubes $\calG(Q)$, such that 
\begin{align*}
O_{\calC}(f_1, Q)
&=\textnormal{ess sup}_{z, z' \in T_Q^{up}} \left|f_{T_{R(z)}^{up}}-f_{T_{R(z')}^{up}} \right| \\
&\lesssim \sum_{R \in \calR(Q)} M_{\calC, p}(f, R)+\sum_{G \in \calG(Q)} M_{\calC, p}(f, G),
\end{align*}
with $\# \calR(Q), \calG(Q) \lesssim_n 1$. Moreover, for any $Q \in \calR(Q) \cap \calG(Q)$, we have
\begin{equation} \label{20260502eq11}
\ell(Q) \in 
\left\{2^{-k-2},\,2^{-k-1},\,2^{-k},\,2^{-k+1}, \frac{3}{4} \cdot 2^{-k-2},\,\frac{3}{4} \cdot 2^{-k-1},\,\frac{3}{4} \cdot 2^{-k},\, \frac{3}{4} \cdot 2^{-k+1} \right\}.
\end{equation} 
Hence, 
\begin{align*}
\calB_{\varepsilon, \calF(\varepsilon)}(f_1)
& =\varepsilon^{n-1} \sum_{Q \in \calF(\varepsilon)} O_{\calC}(f_1, Q) \\
& \lesssim \varepsilon^{n-1} \sum_{Q \in \calF(\varepsilon)} \left( \sum_{R \in \calR(Q)} M_{\calC, p}(f, R)+\sum_{G \in \calG(Q)} M_{\calC, p}(f, G) \right)
\end{align*}
Arguing as in the last part of \textit{Step I}, we may further partition the collection of cubes
\[
\bigcup_{Q \in \calF(\varepsilon)} \left(\calR(Q) \cup \calG(Q) \right)
\]
into subcollections of mutually disjoint cubes. The sidelengths of these cubes belong to the set specified in \eqref{20260502eq11}, and each subcollection has cardinality at most $\varepsilon^{1-n}$. Therefore, 
$$
\calB_{\varepsilon, \calF(\varepsilon)} \lesssim \|f\|_{B_{\calC}^p}
$$
and the desired claim in \textit{Step II} follows by taking the supremum over all $0<\varepsilon \le 1$ and all possible choices for disjoint $\varepsilon$-cubes $\calF(\varepsilon)$. The proof is complete. 
\end{proof}

\begin{rem}
The decomposition in Theorem~\ref{20260502thm41} is motivated by Zhu's decomposition of Bergman-metric BMO spaces into bounded-oscillation and bounded-average parts; see \cite{Zhu1992}. We point out, however, that the present proof is different in nature. In the Bergman setting, one has access to complex-analytic tools such as the Berezin transform and the Bergman projection. In the present Carleson-tent setting, the argument is instead based on elementary geometric properties of upper tents and on finite chains of cubes.
\end{rem}

\bigskip

\section{The Bourgain--Brezis--Mironescu rigidity theorem revisited}\label{Sec4}

We begin this section by recording a straightforward consequence of Theorem \ref{BBMMain}.

\begin{lem} \label{20260608lem02}
Let $1 \le p<+\infty$, $g: Q_0 \to \Z$ be a measurable function, and assume that $F_g \in B_{\calC, 0}^p$, where
$$
F_g(x, t):=g(x), \qquad (x, t) \in Q_0 \times (0, 1).
$$
Then $g$ is a constant. In particular, $F_g$ is also a constant.
\end{lem}

\begin{proof}
Note that
\begin{align*}
\left(F_g \right)_{T_Q^{\textrm{up}}}
&=\frac{1}{\left|T_Q^{\textrm{up}} \right|} \int_{T_Q^{\textrm{up}}} F_g(x, t) dxdt \\
&=\frac{1}{|Q|} \cdot \frac{2}{\ell(Q)} \int_Q \int_{\ell(Q)/2}^{\ell(Q)} g(x) dtdx \\
&=\frac{1}{|Q|} \int_Q g=g_Q.
\end{align*}
Hence,
\begin{align*}
M_{\calC, p}(F_g, Q)
&=\left(\frac{1}{\left|T_Q^{\textrm{up}} \right|} \int_{T_Q^{\textrm{up}}} \left|F_g- \left(F_g \right)_{T_Q^{\textrm{up}}}\right|^p \right)^{1/p} \\
&= \left(\frac{1}{|Q|} \cdot \frac{2}{\ell(Q)} \int_Q \int_{\ell(Q)/2}^{\ell(Q)} \left|g(x)- g_Q\right|^pdtdx \right)^{1/p} \\
&=\left(\intavg_Q |g-g_Q|^p dx \right)^{1/p} \\
& \ge \intavg_Q |g-g_Q| dx,
\end{align*}
where we used H\"older's inequailty in the last step; consequently,
$$
[g]_\varepsilon^{B} \le [F_g]_{\varepsilon, p}.
$$
Thus, if $F_g \in B_{\calC, 0}^p$, namely, $[F_g]_p=\overline{\lim\limits_{\varepsilon \to 0}} [F_g]_{\varepsilon, p}=0$, then $\overline{\lim\limits_{\varepsilon \to 0}} [g]_{\varepsilon}^B=0$, that is $g \in B_0$. The desired claim then follows from Theorem \ref{BBMMain}. 
\end{proof}

Lemma~\ref{20260608lem01} and Lemma~\ref{20260608lem02} capture the two extreme manifestations of the rigidity pheomenon in the Carleson setting: Lemma ~\ref{20260608lem01} states that if $f \in Q_0 \times (0, 1)$ takes constant value along the $x$-direction, then Theorem \ref{BBMMain} fails; while Lemma~\ref{20260608lem02} states that if $f \in Q_0 \times (0, 1)$ takes constant value along the $t$-direction, then Theorem \ref{BBMMain} has in fact a valid counterpart in the Carleson setting. In what follows, we provide a complete understanding of the occurrence of rigidity in Theorem \ref{BBMMain}, but now in the Carleson setting.

\subsection{The $B_{\calC}^p$-trace}

As already alluded above, Lemma~\ref{20260608lem01} and Lemma~\ref{20260608lem02} motivate the following perspective: to capture the essence of Theorem \ref{BBMMain} in the Carleson setting, one should pass to the limit as $t \to 0^+$ and examine whether the resulting limit, if it exists, satisfies the conclusion of Theorem \ref{BBMMain}. This leads us to introduce the following definition.

\begin{defn}
Let $1 \le p<+\infty$ and $f \in B_{\calC}^p$. We say that a measurable function $g$ defined on $Q_0$ is the \emph{$B_{\calC}^p$-trace of $f$}, and denote this by $\textnormal{Tr}_{\calC, p}(f)=g$, if
$$
\lim_{\varepsilon \to 0} \calA_{\varepsilon, p}(f, g)=0, 
$$
where
\begin{equation} \label{20260610eq30}
\calA_{\varepsilon, p}(f, g):=\sup_{\calF(\varepsilon)} \left\{ \varepsilon^{n-1} \sum_{Q \in \calF(\varepsilon)} \left( \intavg_{T_Q^{\textrm{up}}} |f-F_g|^p  \right)^{1/p} \right\}
\end{equation} 
and $F_g(x, t)=g(x)$ is defined as in Lemma \ref{20260608lem02}.
\end{defn}
\begin{rem}
Observe that, in the above definition of the $B_{\calC}^p$-trace, the quantity in \eqref{20260610eq30} is of BBM-bounded-average type; more precisely, it corresponds to the quantity defining the \emph{BBM-bounded average space} $\textrm{BBA}_{\calC}^p(Q_0)$; see \eqref{20260610eq33}. 

In particular, $\calA_{\varepsilon, p}(f, g)$ is \emph{not} defined through the original semi-norm $\|\cdot\|_{B_{\calC}^p}$.
The reason for this is that the $B_{\calC}^p$ semi-norm only measures oscillation modulo local constants, while a trace should measure the actual convergence of $f$ to a fixed boundary profile. Indeed, suppose that one would try to replace \eqref{20260610eq30} by the oscillation-type quantity
$$
\widetilde{\calA}_{\varepsilon,p}(f,g):=
\sup_{\calF(\varepsilon)}
\left\{\varepsilon^{n-1}
\sum_{Q\in\calF(\varepsilon)}
\left(\intavg_{T_Q^{\textrm{up}}}
\left|\left(f-f_{T_Q^{\textrm{up}}}\right)-
\left(F_g-(F_g)_{T_Q^{\textrm{up}}}\right)
\right|^p\,dV\right)^{1/p}\right\}.
$$
Then, the resulting notion of trace would be too weak and, in particular, would not give a well-defined trace map. For instance, taking $f=0$ and letting $g$ be any nonconstant smooth function on $Q_0$, since $F_g$ is smooth in the horizontal variables and independent of $t$, we would obtain
$$
\widetilde{\calA}_{\varepsilon,p}(0,g)
=\sup_{\calF(\varepsilon)}
\left\{\varepsilon^{n-1} \sum_{Q\in\calF(\varepsilon)}
M_{\calC,p}(F_g,Q)\right\}
\longrightarrow 0
\qquad \textrm{as } \varepsilon\to0.
$$
Thus, under this oscillation-type definition, such a nonconstant function $g$ would be regarded as a trace of the zero function $f=0$, which is absurd. Equivalently, the trace would fail to be unique.

The same obstruction remains even if one works in the quotient space $B_{\calC}^p/\mathbb R$. Namely, passing to the quotient space only removes the ambiguity caused by global constants, whereas the $B_{\calC}^p$ semi-norm removes local constants on each tent. Hence, the trace must be defined through the BBM-bounded-average type convergence in \eqref{20260610eq30}, rather than the $B_{\calC}^p$ oscillation semi-norm.
\end{rem}

Having defined the notion of $B_{\calC}^p$-trace, we show first that, when it exists, the trace is unique and is locally $p$-integrable.

\begin{lem} \label{20260610lem01}
Let \(1 \le p<+\infty\) and \(f \in B_{\calC}^p\). If \(\textnormal{Tr}_{\calC,p}(f)\) exists, then it is unique and belongs to $L^p(Q_0)$.
\end{lem}

\begin{proof}
We first show that if $\textnormal{Tr}_{\calC, p}(f)$ exists, then it is unique. For any measurable function \(u\) on \(Q_0\), define
\[
\Lambda_{\varepsilon,p}(u)
:=
\sup_{\calF(\varepsilon)}
\left\{
\varepsilon^{n-1}
\sum_{Q \in \calF(\varepsilon)}
\left( \intavg_Q |u(x)|^p\,dx \right)^{1/p}
\right\}.
\]
Assume that both \(g_1\) and \(g_2\) are \(B_{\calC}^p\)-traces of \(f\). Then, by the triangle inequality,
\[
\Lambda_{\varepsilon,p}(g_1-g_2)
\le
\calA_{\varepsilon,p}(f,g_1)+\calA_{\varepsilon,p}(f,g_2).
\]
By assumption, the right-hand side converges to zero as \(\varepsilon \to 0\). Hence
\begin{equation} \label{20260609eq10}
\lim_{\varepsilon \to 0}\Lambda_{\varepsilon,p}(g_1-g_2)=0.
\end{equation}
Denote \(u:=g_1-g_2\). It remains to show that \eqref{20260609eq10} implies
\begin{equation} \label{20260609eq11}
u=0 \quad \text{a.e. on } Q_0.
\end{equation}
Assume the contrary. Then there exists some \(\delta>0\) such that \(|E_\delta|>0\), where
\[
E_\delta:=\left\{x \in Q_0:\ |u(x)|>\delta \right\}.
\]
By the Lebesgue density theorem and using a Vitali covering type argument, for all sufficiently small \(\varepsilon>0\), one can find a collection of \(\varepsilon\)-cubes $\mathcal Q:=\{Q_j\}$
with \(\#\mathcal Q\gtrsim \varepsilon^{1-n}\), such that for every \(Q_j\in\mathcal Q\) we have
\[
|E_\delta \cap Q_j| \ge \frac{1}{2}|Q_j|.
\]
Consequently, for each \(Q_j\in\mathcal Q\), we obtain
\[
\left(\intavg_{Q_j} |u(x)|^p\,dx \right)^{1/p}
\ge
\frac{\delta}{2^{1/p}}.
\]
Therefore,
\[
\Lambda_{\varepsilon,p}(u)
\ge
\varepsilon^{n-1}
\sum_{Q_j \in \mathcal Q}
\left(\intavg_{Q_j} |u(x)|^p\,dx \right)^{1/p}
\gtrsim \delta,
\]
which contradicts \eqref{20260609eq10}. Thus \(u=0\) a.e. on \(Q_0\), and the uniqueness of the $B_{\calC}^p$-trace is proved. 

\vspace{0.1cm}

Next, assume $g=\operatorname{Tr}_{\mathcal C,p}(f)$ exists. We aim to show that
$g\in L^p(Q_0)$. Indeed, let $\varepsilon_0>0$ (fixed) be such that
\begin{equation} \label{20260610eq02}
\mathcal A_{\varepsilon_0,p}(f,g)<1.
\end{equation} 
Without loss of generality, let us further assume that $\varepsilon_0^{-1}\in\mathbb N$. Divide
$Q_0=[0,1)^n$ into the standard family of $\varepsilon_0$-cubes. The
total number of such cubes is $\varepsilon_0^{-n}$. We then split this
family into $\varepsilon_0^{-1}$ layers
$\mathcal F_j(\varepsilon_0)$, $j=1,\ldots,\varepsilon_0^{-1}$, by grouping together those $\varepsilon_0$-cubes with the same position in one fixed coordinate direction. In this way, each $\mathcal F_j(\varepsilon_0)$
consists of mutually disjoint $\varepsilon_0$-cubes and $\#\mathcal F_j(\varepsilon_0)=\varepsilon_0^{1-n}$. Using these collections with \eqref{20260610eq02} yields
$$
\int_{Q_0 \times (\varepsilon_0/2, \varepsilon_0)} |f-F_g|^p \lesssim_{\varepsilon_0} 1.
$$
This together with the assumption that $f \in L^p(Q_0 \times (0, 1))$ gives $\int_{Q_0 \times (\varepsilon_0/2, \varepsilon_0)} |F_g|^p \lesssim_{\varepsilon_0} 1$, which implies $g \in L^p(Q_0)$.
\end{proof}

Next, we address the issue of existence of $B_{\calC}^p$-trace. The following example may be viewed as a variant of the simple example considered in Lemma \ref{20260608lem01}.

\begin{lem} \label{20260610lem02}
Let $1 \le p<+\infty$. There exists a function $f \in B_{\calC,0}^p$ such that $\textnormal{Tr}_{\calC,p}(f)$ does not exist.
\end{lem}

\begin{proof}
Consider
\begin{equation} \label{20260610eq11}
f(x,t)=h(t),\qquad \text{where} \qquad h(t)=\sin\left(\log\log \frac{e}{t}\right),\quad 0<t<1.
\end{equation} 
Since $h$ is bounded, it follows immediately that $f\in B_{\calC}^p$. In fact, we have that $f\in B_{\calC,0}^p$. Indeed, for each $\varepsilon$-cube $Q$, one has
\begin{align} \label{20260610eq01}
M_{\calC,p}(f,Q)
&=\left(
\frac{2}{\varepsilon}\int_{\varepsilon/2}^{\varepsilon}
\left|h(t)-\frac{2}{\varepsilon}\int_{\varepsilon/2}^{\varepsilon}h(\widetilde t)\,d\widetilde t\right|^p dt
\right)^{1/p} \nonumber \\
&\le \operatorname{osc}_{t\in[\varepsilon/2,\varepsilon]} h(t)
\le \int_{\varepsilon/2}^{\varepsilon}|h'(t)|\,dt \nonumber \\
&\le \int_{\varepsilon/2}^{\varepsilon}
\frac{1}{t\log(e/t)}\,dt
=\log\left(
\frac{\log(2e/\varepsilon)}{\log(e/\varepsilon)}
\right)
\longrightarrow 0,
\qquad \text{as } \varepsilon\to0.
\end{align}
Therefore,
$$
\lim_{\varepsilon\to0}\,[f]_{\varepsilon,p}
=\lim_{\varepsilon\to0}\sup_{\calF(\varepsilon)}
\left\{\varepsilon^{n-1}\sum_{Q\in\calF(\varepsilon)}
M_{\calC,p}(f,Q)\right\}=0,
$$
which proves that $f\in B_{\calC,0}^p$.

\vspace{0.1cm}

Next, we show that $\textnormal{Tr}_{\calC,p}(f)$ does not exist. Assume the contrary, that is, 
$$
\textnormal{Tr}_{\calC,p}(f)=g
$$
for some measurable function $g$ on $Q_0$. Since $h$ oscillates infinitely many times near $t=0$, we may choose two sequences $\{\varepsilon_k\}_{k\ge1}$ and $\{\eta_k\}_{k\ge1}$ such that
$$
\lim_{k\to\infty}\varepsilon_k=\lim_{k\to\infty}\eta_k=0,
$$
and
$$
h(\varepsilon_k)=1,\qquad h(\eta_k)=-1,\qquad k\ge1.
$$
Then, by the triangle inequality, for each $k\ge1$, we can write
\begin{align*}
\Lambda_{\varepsilon_k,p}(g-1)
&=\Lambda_{\varepsilon_k,p}(g-h(\varepsilon_k)) \\
&\le \calA_{\varepsilon_k,p}(f,g)+\operatorname{osc}_{t\in[\varepsilon_k/2,\varepsilon_k]}h(t).
\end{align*}
Letting $k\to\infty$, and using both the assumption $g=\textnormal{Tr}_{\calC,p}(f)$ and estimate \eqref{20260610eq01}, we obtain
$$
\lim_{k\to\infty}\Lambda_{\varepsilon_k,p}(g-1)=0.
$$
The same argument used to prove \eqref{20260609eq11} in the previous lemma also applies along the sequence $\{\varepsilon_k\}_{k\geq 1}$. Hence, we readily get 
$$
g=1 \quad \text{a.e. on } Q_0.
$$
On the other hand, applying the same reasoning to the sequence $\{\eta_k\}_{k\ge1}$, we get
$$
\lim_{k\to\infty}\Lambda_{\eta_k,p}(g+1)=0,
$$
and hence
$$
g=-1 \quad \text{a.e. on } Q_0.
$$
This is a contradiction. Therefore, $\textnormal{Tr}_{\calC,p}(f)$ does not exist.
\end{proof}

Lemma~\ref{20260610lem01} and Lemma~\ref{20260610lem02} motivate us to introduce the following subspace of $B_{\calC}^p$. 

\begin{defn}
Let $1 \le p<+\infty$. Denote the \emph{traceable subspace} of $B_{\calC}^p$ by
$$
T_{\calC}^p:=\left\{f \in B_{\calC}^p: \exists \; g \in L^p(Q_0) \ \textrm{such that} \ \textnormal{Tr}_{\calC, p}(f)=g \right\}. 
$$
\end{defn}

\noindent In particular, Lemma \ref{20260610lem02} shows that
$$
T_{\calC}^p \subsetneq B_{\calC}^p, \qquad \textrm{and} \qquad  B_{\calC, 0}^p \not\subseteq T_{\calC}^p.
$$
Moreover, Lemma \ref{20260610lem01} asserts that the trace map
$$
\textnormal{Tr}_{\calC,p}:T_{\calC}^p\to L^p(Q_0)
$$
is well-defined.

\subsection{Properties of the traceable subspace $T_{\calC}^p$}

In this section, we present several structural results for the function space $T_{\calC}^p$. First, motivated by Lemma \ref{20260610lem02}, it is natural to ask whether $T_{\calC}^p$ is contained in $B_{\calC, 0}^p$. We have the following.

\begin{lem} \label{20260610eq40}
For any $1 \le p<+\infty$, we have
$
T_{\calC}^p \not\subseteq B_{\calC,0}^p.
$
\end{lem}

\begin{proof}
Let
$$
E:=\left\{x\in Q_0:\ 0\le x_1<1/2\right\}, \qquad g(x):=\one_E(x),
$$
and let $F_g$ be the extension of $g$ as usual. Since $\calA_{\varepsilon,p}(F_g,g)=0$ for every $0<\varepsilon \le 1$, it follows immediately from the definition of trace that $F_g\in T_{\calC}^p$ and $ \textnormal{Tr}_{\calC,p}(F_g)=g$.

It remains to show that $F_g\notin B_{\calC,0}^p$. For every $\varepsilon>0$ (sufficiently small), one can find a collection $\calQ_\varepsilon$ of mutually disjoint $\varepsilon$-cubes in $Q_0$, with
$
\#\calQ_\varepsilon \gtrsim \varepsilon^{1-n},
$
such that the hyperplane $\{x_1=1/2\}$ divides each $Q\in\calQ_\varepsilon$ into two parts of equal volume; see, Figure \ref{Fig1}. 

\begin{figure}[htbp]
\centering
\begin{tikzpicture}[scale=4.8, every node/.style={font=\small}]

\def\eps{0.13}

\draw[thick] (0,0) rectangle (1,1);

\draw[->, thick] (0,0) -- (1.2,0);
\draw[->, thick] (0,0) -- (0,1.2);

\draw[very thick, dashed] (0.5,0) -- (0.5,1);
\node[above] at (0.5,1.04) {$\{x_1=1/2\}$};

\foreach \k in {0,...,5} {
    \pgfmathsetmacro{\y}{0.08+\k*0.145}
    \fill[gray!20] (0.5-\eps/2,\y) rectangle (0.5+\eps/2,\y+\eps);
    \draw[thick] (0.5-\eps/2,\y) rectangle (0.5+\eps/2,\y+\eps);
    \draw[densely dotted] (0.5,\y) -- (0.5,\y+\eps);
}

\draw[<->] (0.5-\eps/2-0.035,0.08) -- (0.5-\eps/2-0.035,0.08+\eps);
\node[left, font=\scriptsize] at (0.5-\eps/2-0.04,0.08+0.5*\eps) {$\varepsilon$};

\node[below left] at (0,0) {$0$};
\node[below] at (0.5,0) {$1/2$};
\node[below] at (1,0) {$1$};
\node[left] at (0,1) {$1$};

\node[below] at (1.16,0) {$x_1$};
\node[left] at (0,1.16) {$(x_2,\dots,x_n)$};

\node[anchor=north west] at (0.02,0.98) {$Q_0$};
\node[right] at (0.62,0.52) {$Q\in\mathcal Q_\varepsilon$};

\end{tikzpicture}
\caption{A family $\mathcal Q_\varepsilon$ of mutually disjoint $\varepsilon$-cubes in $Q_0$, each bisected by the hyperplane $\{x_1=1/2\}$.}
\label{Fig1}
\end{figure}
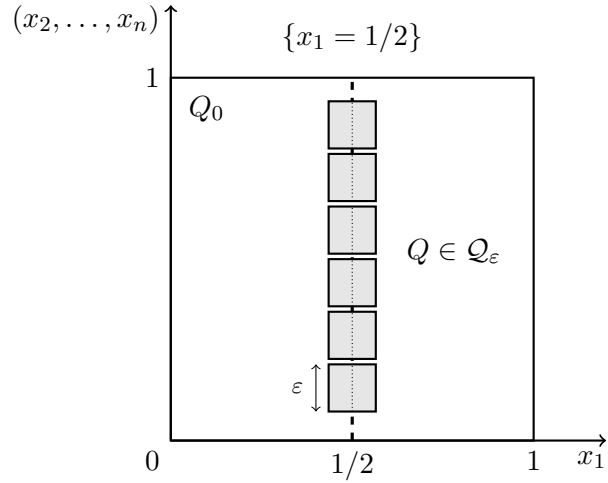

Now, for each such cube $Q$, since $F_g$ is independent of $t$, its average over $T_Q^{up}$ is equal to $1/2$. Hence,
$$
M_{\calC,p}(F_g,Q)
=
\left(\intavg_{T_Q^{up}}\left|F_g-\frac12\right|^p\,dV\right)^{1/p}
=
\frac12.
$$
Therefore,
$$
[F_g]_{\varepsilon,p}
\ge
\varepsilon^{n-1}
\sum_{Q\in\calQ_\varepsilon}
M_{\calC,p}(F_g,Q)
\gtrsim
\varepsilon^{n-1}\varepsilon^{1-n}
\gtrsim 1.
$$
Consequently,
$$
\liminf_{\varepsilon\to0}\,[F_g]_{\varepsilon,p}>0,
$$
which implies $F_g\notin B_{\calC,0}^p$.
\end{proof}

A natural next question is to ask whether $T_{\calC}^p$ is closed under the topology induced by the space $B_{\calC}^p$. However, since $\| \cdot \|_{B_{\calC}^p}$ defines only a semi-norm on $B_{\calC}^p$, the more meaningful question to ask is about the closure in the equivalence classes. More precisely, consider the quotient spaces
$$
\dot{B}_{\calC}^p:=B_{\calC}^p /\R, \qquad \textnormal{and} \qquad \dot{T}_{\calC}^p:=T_{\calC}^p /\R.
$$
Moreover, in this case, we interpret the trace map acting on these quotient spaces, namely
$$
\textnormal{Tr}_{\calC, p}: \dot{T}_{\calC}^p \longrightarrow \dot{L}^p(Q_0):=L^p(Q_0)/\R, \qquad [f] \longrightarrow \left[\textnormal{Tr}_{\calC, p}(f) \right ]
$$
Observe that, by Lemma \ref{20260610lem01}, the above mapping is also well-defined. 

\vspace{0.1cm}

We have the following result.

\begin{lem} \label{20260610lem65}
For any $1 \le p<+\infty$, $\dot{T}_{\calC}^p$ is not closed in the norm topology induced by $\|\cdot \|_{B_{\calC}^p}$ on $\dot{B}_{\calC}^p$.
\end{lem}

\begin{proof}
We show that the result is a consequence of the construction in Lemma \ref{20260610lem02}. Recall that in \eqref{20260610eq11} we have constructed the function
$$
f(x, t)=h(t) \qquad \textrm{with} \qquad h(t)=\sin \left(\log \log \frac{e}{t} \right), \quad 0<t<1,
$$
and we have showed that $f \in B_{\calC, 0}^p \backslash T_{\calC}^p$. It is clear that the same argument used to prove Lemma \ref{20260610lem02} also guarantees that $f+c \notin T_{\calC}^p$ for any $c \in \R$, and hence $[f] \notin \dot{T}_{\calC}^p$.

\vspace{0.1cm}

Next, we shall construct a sequence $\{[f_m]\}_{m \ge 1} \subseteq \dot{T}_{\calC}^p$ such that 
\begin{equation} \label{20260619eq30}
[f_m] \to [f]  \qquad \textrm{as} \qquad m \to \infty,
\end{equation} 
in the sense of the norm topology on $\dot{B}_{\calC}^p$. Indeed, for each $m \ge 2$, define 
$$
h_m(t):=
\begin{cases}
h \left(\frac{1}{m} \right), \qquad \hfill 0<t<\frac{1}{m}; \\
\\
h(t), \qquad \hfill \frac{1}{m} \le t<1,
\end{cases}
$$
and let $f_m(x, t):=h_m(t)$ be its extension. It is immediate that, for each $m \ge 2$, $f_m \in T_{\calC}^p$ with $\textnormal{Tr}_{\calC, p}(f_m)=h(1/m)$. This gives, for each $m \ge 1$, $\textnormal{Tr}_{\calC, p} \left([f_m] \right)=[0]$, and hence $\{[f_m]\}_{m \ge 1} \subseteq \dot{T}_{\calC}^p$. We are left to prove \eqref{20260619eq30}. Set
$$
r_m:=f_m-f.
$$
Then
$$
r_m(x,t):=
\begin{cases}
h\left(\frac{1}{m} \right)-h(t), \qquad \hfill 0<t<\frac{1}{m}; \\
\\
0, \qquad \hfill \frac{1}{m}\le t<1.
\end{cases}
$$
We claim that
$$
\|[f_m]-[f]\|_{B_{\calC}^p}
=
\sup_{0<\varepsilon \le 1}[r_m]_{\varepsilon,p}
\to 0 \qquad \textrm{as} \qquad m \to \infty.
$$
Fix $0<\varepsilon \le 1$ and let $Q$ be any $\varepsilon$-cube. Since $r_m$ only depends on $t$, one has
$$
M_{\calC,p}(r_m,Q)
\le \textnormal{osc}_{t\in[\varepsilon/2,\varepsilon]}r_m(t).
$$
We split the analysis into three cases. 

\medskip 

\noindent $\bullet$ If $\varepsilon\ge 2/m$, then $[\varepsilon/2,\varepsilon]\subset [1/m,1)$, and hence $r_m\equiv0$ on $T_Q^{\textrm{up}}$. Thus, 
$$
M_{\calC,p}(r_m,Q)=0.
$$

\medskip

\noindent $\bullet$ If $1/m\le \varepsilon<2/m$, then the only possible variation of $r_m$ occurs on $[\varepsilon/2,1/m]\subset[1/(2m),1/m]$. Therefore,
$$
M_{\calC,p}(r_m,Q) \le
\textnormal{osc}_{t\in[1/(2m),1/m]}h(t).
$$

\medskip 

\noindent $\bullet$ If $0<\varepsilon<1/m$, then $[\varepsilon/2,\varepsilon]\subset(0,1/m)$, and hence
$$
M_{\calC,p}(r_m,Q)
\le
\operatorname{osc}_{t\in[\varepsilon/2,\varepsilon]}h(t).
$$

\medskip 

Combining the three cases above with the argument in \eqref{20260610eq01}, we obtain
$$
\sup_{0<\varepsilon \le 1}\sup_{\ell(Q)=\varepsilon}
M_{\calC,p}(r_m,Q)
\to 0, \qquad \textrm{as} \qquad m \to \infty, 
$$
which gives
$$
[f_m]\to[f]
\qquad \textrm{in } \dot{B}_{\calC}^p.
$$
The proof of \eqref{20260619eq30} is complete. 
\end{proof}

Our next result shows that although neither of the inclusions
$$
B_{\calC,0}^p\subset T_{\calC}^p
\qquad \textrm{nor} \qquad
T_{\calC}^p\subset B_{\calC,0}^p
$$
holds in general, the spaces $B_{\calC,0}^p$ and $T_{\calC}^p$ are still closely related after passing to the quotient space.

\begin{prop}\label{20260619lem03}
Let $1\le p<+\infty$. Define
$$
\mathcal V_{\calC}^p
:=
\left\{
[F_g]\in \dot B_{\calC}^p:\ g\in L^p(Q_0),\ F_g(x,t):=g(x)
\right\}.
$$
Then
$$
\dot T_{\calC}^p\subset \mathcal V_{\calC}^p+\dot B_{\calC,0}^p.
$$
Moreover,
$$
\dot T_{\calC}^p+\dot B_{\calC,0}^p
=
\mathcal V_{\calC}^p+\dot B_{\calC,0}^p.
$$
\end{prop}

\begin{proof}
Let $[f]\in \dot T_{\calC}^p$, and let
$$
g:=\textnormal{Tr}_{\calC,p}(f).
$$
Then, by the definition of the $B_{\calC}^p$-trace,
$$
\lim_{\varepsilon\to0}\mathcal A_{\varepsilon,p}(f,g)=0.
$$
We first observe that $F_g\in B_{\calC}^p$. Indeed, for every cube $Q\subset Q_0$, by the triangle inequality,
$$
M_{\calC,p}(F_g,Q)
\lesssim M_{\calC,p}(f,Q)+
\left(\intavg_{T_Q^{\textrm{up}}}
|f-F_g|^p\,dV \right)^{1/p}.
$$
Hence, after summing over all admissible families $\calF(\varepsilon)$ and taking the supremum, one obtains
$$
[F_g]_{\varepsilon,p}
\lesssim
[f]_{\varepsilon,p}
+\mathcal A_{\varepsilon,p}(f,g).
$$
Since $f\in B_{\calC}^p$ and $\mathcal A_{\varepsilon,p}(f,g)$ is bounded for $0<\varepsilon \le 1$, it follows that $F_g\in B_{\calC}^p$.

\vspace{0.1cm}

Now set
$$
u:=f-F_g.
$$
Then $u\in B_{\calC}^p$. Moreover, for every cube $Q\subset Q_0$,
$$
M_{\calC,p}(u,Q)
\lesssim \left(
\intavg_{T_Q^{\textrm{up}}}
|u|^p\,dV
\right)^{1/p}.
$$
Therefore,
$$
[u]_{\varepsilon,p}
\lesssim \mathcal A_{\varepsilon,p}(f,g)
\to 0
\qquad \textrm{as } \varepsilon\to0,
$$
which gives $u\in B_{\calC,0}^p$. Consequently,
$$
[f]=[F_g]+[u]\in \mathcal V_{\calC}^p+\dot B_{\calC,0}^p,
$$
which proves
\begin{equation} \label{20260610eq45}
\dot T_{\calC}^p\subset \mathcal V_{\calC}^p+\dot B_{\calC,0}^p.
\end{equation} 

\medskip

On the other hand, for every $[F_g]\in \mathcal V_{\calC}^p$, one has $\mathcal A_{\varepsilon,p}(F_g,g)=0$ for all $0<\varepsilon \le 1$, and hence $F_g\in T_{\calC}^p$ with $\textnormal{Tr}_{\calC,p}(F_g)=g$.
Therefore,
\begin{equation} \label{20260610eq46}
\mathcal V_{\calC}^p\subset \dot T_{\calC}^p.
\end{equation} 
Combining the two inclusions \eqref{20260610eq45} and \eqref{20260610eq46} gives
$$
\dot T_{\calC}^p+\dot B_{\calC,0}^p
=
\mathcal V_{\calC}^p+\dot B_{\calC,0}^p.
$$
This completes the proof for the second inclusion.
\end{proof}

\begin{rem}\label{20260619ques01}
Lemma \ref{20260610lem65} raises the following natural question: can one characterize the closure
$
\overline{\dot T_{\calC}^p}^{,\dot B_{\calC}^p}?
$
This question appears to be of independent interest, but its answer eludes us for now.
\end{rem}

\subsection{The rigidity theorem in the Carleson setting}

Finally, we return to the original motivation, namely Theorem \ref{BBMMain}, in the Carleson setting. Our main result in this direction is captured by the result below, which constitutes a natural application of the classical BBM-rigidity result.

\begin{thm} \label{20260610BBMCarleson}
Let $1 \le p<+\infty$ and $f \in T_{\calC}^p \cap B_{\calC, 0}^p$ with $g=\textnormal{Tr}_{\calC, p}(f)$. If $g$ is integer-valued, then $g$ is constant $a.e.$ on $Q_0$. 
\end{thm}

\begin{proof}
As usual, denote by $F_g(x, t)$ the extension of $g$. Since $F_g$ does not depend on $t$,  $(F_g)_{T_Q^{\textnormal{up}}}=g_Q$. Hence, for any $\varepsilon$-cube $Q \subseteq Q_0$, we have
\begin{align*}
\intavg_Q \left|g-g_Q\right| dx 
&\le \left(\intavg_Q \left|g-g_Q\right|^p dx \right)^{1/p} \\
& =M_{\calC, p}(F_g, Q) \\
& \lesssim M_{\calC, p}(f, Q)+ \left(\intavg_{T_Q^{up}} |f-F_g|^p \right)^{1/p}.
\end{align*}
Thus, for any admissible family $\calF(\varepsilon)$, we can write
\begin{align*}
\varepsilon^{n-1} \sum_{Q \in \calF(\varepsilon)} \intavg_Q \left|g-g_Q\right| dx  
& \lesssim \varepsilon^{n-1} \sum_{Q \in \calF(\varepsilon)} M_{\calC, p}(f, Q)  \\
& \quad +\varepsilon^{n-1} \sum_{Q \in \calF(\varepsilon)}\left(\intavg_{T_Q^{up}} |f-F_g|^p \right)^{1/p},
\end{align*}
which further gives
$$
[g]^B_{\varepsilon} \le [f]_{\varepsilon, p}+\calA_{\varepsilon, p}(f, g),
$$
with $\calA_{\varepsilon, p}(f, g)$ as defined in \eqref{20260610eq30}. Now, since $f \in B_{\calC, 0}^p$, we have 
$$
\lim_{\varepsilon \to 0}\, [f]_{\varepsilon, p}=0.
$$
Moreover, since $g=\textnormal{Tr}_{\calC, p}(f)$, 
$$
\lim_{\varepsilon \to 0} \calA_{\varepsilon, p}(f, g)=0.
$$
We conclude that
$$
\lim_{\varepsilon \to 0}\, [g]_{\varepsilon}^B=0,
$$
that is, $g \in B_0$. Therefore, by Theorem \ref{BBMMain}, $g$ must be a constant. 
\end{proof}

\end{document}